\documentclass[12pt]{article}
\usepackage{amsmath}
\usepackage{graphicx,psfrag,epsf}
\usepackage{enumerate}
\usepackage{natbib}
\usepackage{url} 

\usepackage{epigraph} 
\usepackage{amssymb}
\usepackage{epstopdf}
\usepackage{subcaption}
\usepackage{caption}
\usepackage{authblk,lineno,amsthm}

\newtheorem{prop}{Proposition}
\newtheorem{corr}{Corollary}
\def\PP{\mathbf{P}}
\def\EE{\mathbf{E}}
\def\Var{\mathbf{Var}}
\def\Cov{\mathbf{Cov}}
\def\sign{\mathrm{sign}}

\def\toP{\overset{\PP}{\longrightarrow}}
\def\toD{\overset{d}{\longrightarrow}}

\newcommand{\blind}{0}

\begin{document}

\def\spacingset#1{\renewcommand{\baselinestretch}%
{#1}\small\normalsize} \spacingset{1}


\if0\blind
{
  \title{\bf High Dimensional Space Oddity}
\author[1]{Haim Bar}
\author[1]{Vladimir Pozdnyakov}
\affil[1]{Department of Statistics, University of Connecticut}
  \maketitle
} \fi

\if1\blind
{
  \bigskip
  \bigskip
  \bigskip
  \begin{center}
    {\LARGE\bf High Dimensional Space Oddity}
\end{center}
  \medskip
} \fi

\bigskip
\begin{abstract}
  In his 1996 paper, Talagrand highlighted that the Law of Large Numbers (LLN)
  for independent random variables can be viewed as a geometric property of
  multidimensional product spaces. This phenomenon is known as the concentration
  of measure. To illustrate this profound connection between geometry and
  probability theory, we consider a seemingly intractable geometric problem in
  multidimensional Euclidean space and solve it using standard probabilistic
  tools such as the LLN and the Central Limit Theorem (CLT).
\end{abstract}

\noindent%
{\it Keywords:}  Central Limit Theorem, Law of Large Numbers, Multidimensional Euclidean Space,
Measure Concentration.
\vfill

\newpage
\spacingset{1.45} 
\section{Introduction: A Curious Result}

\begin{epigraphs}
\qitem{{\it It is through science that we prove, but through intuition that we discover.}}%
      {Henri Poincaré}
\qitem{{\it The only real valuable thing is intuition.}}%
      {Albert Einstein}
\end{epigraphs}

Maybe. Or maybe, intuition is what made Poincaré reject Cantor's set theory, saying that
``There is no actual infinity'', and caused Einstein to oppose the notion of entanglement in quantum
physics, suggesting that it implies ``spooky action at a distance''.

In probability and statistics, there are many counterexamples \citep{WiseHall1993} which show that one
must always be careful trusting one's intuition. Famously, even the great Paul Erdős could not
believe the correct solution to the Monty Hall problem \citep{gardner1982aha}.
But maybe all these examples of failed intuition only occur in relatively recent theories? Maybe our geometric
intuition is sound? This is definitely not the case in dimension greater than 3.
The following example was given in \cite{steele2004} as a cautionary tale, and is the starting
point for this paper.

\cite{steele2004} describes an arrangement of $n$-dimensional balls centered around the
$2^n$ vertices of a cube, $\mathbf{c} \in \{-1, 1\}^n$, each with radius 1,
as in Figure \ref{Mike} for $n=2$, and in the left panel of Figure \ref{3dballs} for $n=3$.
The arrangement is bounded by the cube $[-2, 2]^n$.
The question is, if we enclose a hypersphere centered at the origin so that it is tangent to the hyperspheres
in the arrangement, is it in the cube $[-2, 2]^n$ for all $n \ge 2$?

\begin{figure}[htbp]
\begin{center}
\includegraphics[scale=1]{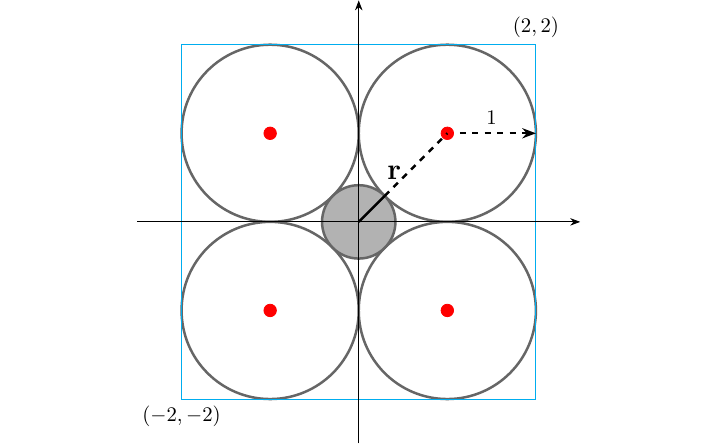}
\caption{An arrangement of spheres as in \cite{steele2004}. The distance from the origin to the center of each ball is $r=\sqrt{n}$.}
\label{Mike}
\end{center}
\end{figure}

The answer is no. To see that, all we need is the multidimensional version of the Pythagorean theorem.
The radius of the hyperspheres is 1, regardless of the dimension, $n$. The length of the line between the center
of a hypersphere and the origin is $\sqrt{n}$, so the radius of the inner hypersphere is $\sqrt{n}-1$. But, for $n>9$,
the inner hypersphere must reach outside the $[-2, 2]^n$ cube! Furthermore, as $n\to\infty$, most of the volume
of the inner hypersphere is outside the cube.

The lesson from this example is that intuition is important for innovation, but we must not
forget the first half of Poincaré's maxim  -- we must do the math to prove or disprove
conjectures before they can be considered discoveries.

Note that from here forward we dispense with the ``hyper'' prefix and just use the term sphere
(or $n$-dimensional sphere) and the notation $S^n(r)$ to describe the surface of a ball of radius $r$ in $\mathbb{R}^n$.
We note that, topologically, the sphere has a dimension of $n-1$, but in order to avoid confusion
we refer to the dimension of Euclidean space in which the sphere is embedded.

\section{A Related Question}

Assume that we put a source of light at the origin. What fraction of light will be blocked by the
unit balls located at vertices of the $[-1, 1]^n$ cube?
It is obvious that nothing will get out for $n=1$ and $n=2$ (see Figure \ref{Mike} for $n=2$). However,
when $n=3$, the collection of $2^3=8$ balls will allow some light out. To see this, just take a look
at the structure along any axis (as can be seen in the left panel of Figure \ref{3dballs}.)

Now, before we go further, we need to provide a formal description of what we mean by the ``fraction of light''.
Imagine an $n$-dimensional sphere that circumscribes the cube $[-1,1]^n$.
Consider all the lines that go through the origin. Lines that intersect any of the unit balls can not reach
the surface of the $[-2,2]^n$ cube, which encloses all the unit balls in the arrangement.
The intersection of those lines with the sphere that circumscribes the $[-1,1]^n$ cube
is the set that represents the shadows of the $2^n$
unit balls on the sphere. Our goal is to compute the ratio of the area
of the total shadow to the area of the entire sphere.
This is depicted in the right panel of Figure \ref{3dballs}, which shows a 2-D projection of a unit ball from a high-dimensional arrangement, centered at one of the $[-1,1]^n$ cube's vertices at a distance of $r=\sqrt{n}$ from the origin, and  rays of light emanating from the origin. The places in the $[-2,2]^n$ cube to which
light can not reach are shaded. The segment defined by the thick arc shows a 2-D projection of the (spherical) cap on the sphere $S^n(\sqrt{n})$
that circumscribes the cube $[-1,1]^n$.
Our question about ``fraction of light'' should be interpreted as finding the total measure of all $2^n$ such caps, relative to the area of the sphere $S^n(\sqrt{n})$.

%

\begin{figure}[htbp]
\centering\begin{subfigure}{0.45\textwidth}
\includegraphics[width=\textwidth, scale=0.7]{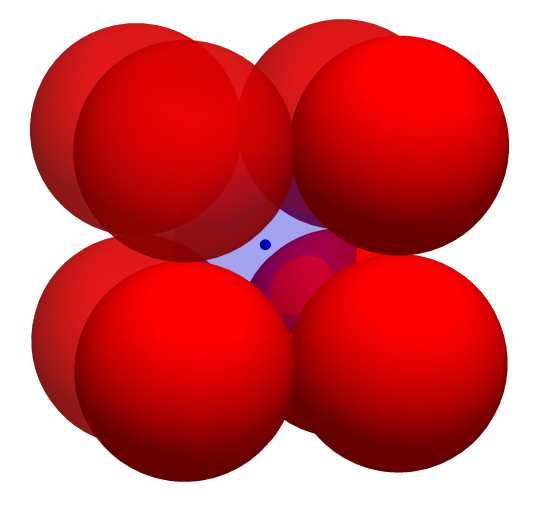}
\end{subfigure}
\begin{subfigure}{0.43\textwidth}
\includegraphics[width=\textwidth, scale=1.3]{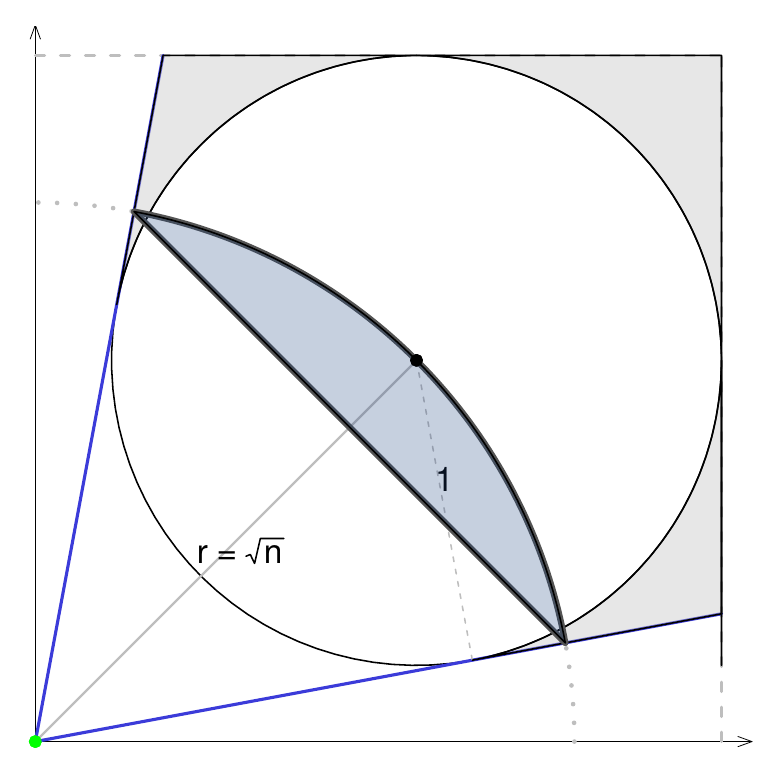}
\label{shadow}
\end{subfigure}
\caption{Left: Arrangement of 8 unit balls in $\mathbb{R}^3$, centered at $\{\pm1, \pm1, \pm 1\}$
and a source of light emanating from the origin. Right: A projection of a high dimensional arrangement, showing the shaded area in the cube, and the shaded cap in the sphere passing through the center of the balls.}\label{3dballs}
\end{figure}


\section{The Shadow of One Ball}

Consider the ball of radius $0\leq r<\sqrt{n}$ located at vertex $(1,\dots,1)$. What is the area of the shadow
of this one ball? The answer is known. The shadow formed by the ball is a spherical cap, which is
a nonempty intersection of a half-space and a sphere.

A cap can be specified in many different ways.  For instance, in \cite{ball1997elementary} caps are described in
terms of a distance to the origin or a cap radius (see Figure \ref{cap}). We will use an angle between the diagonal
vector $(1,\dots,1)$ and the ``edge'' of the cap as it is done in \cite{li2011}, where a nice
concise formula for the area of a cap is derived.

\begin{center}
\end{center}
\begin{figure}[htbp]
  \begin{center}
  \includegraphics[scale=0.5]{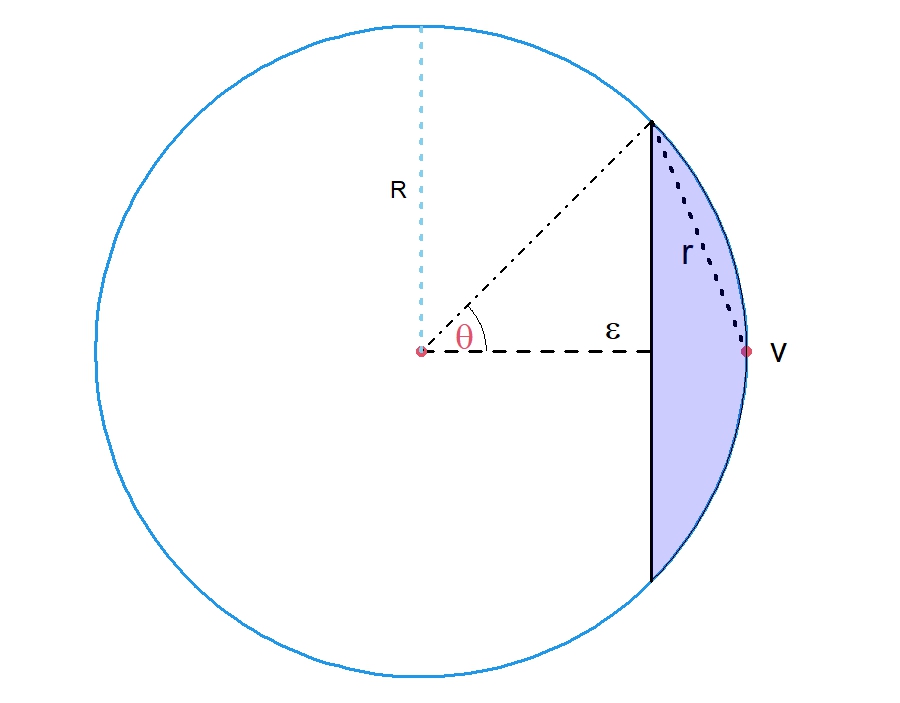}
  \caption{Definitions of a spherical cap centered at $\mathbf{v}$.}
  \label{cap}
  \end{center}
  \end{figure}

More specifically, let us consider $n$-dimensional sphere of radius $R$, denoted by $S^{n}(R)$.
Then for any $0\leq \epsilon \leq R$, $0\leq r\leq\sqrt{R}$ and $0\leq\theta\leq\pi/2$, the smaller cap
can defined as
$$C_d(\epsilon, \mathbf{v})= \{\mathbf{x}\in S^{n}(R) : \mathbf{x}\cdot\mathbf{v} \ge \epsilon\},$$
or
$$C_r(r, \mathbf{v})= \{\mathbf{x}\in S^{n}(R) : ||\mathbf{x} - \mathbf{v}|| \le r\},$$
or
$$C_a(\theta, \mathbf{v})= \{\mathbf{x}\in S^{n}(R) :\arccos(\mathbf{x}\cdot\mathbf{v}) \le \theta\},$$
where $\mathbf{x}\cdot\mathbf{v}$ is the inner product of two vectors, and $||\cdot||$ is
the Euclidean distance. A simple calculation shows that when $r=2R\sin(\theta/2)$ and
$\epsilon=R\cos(\theta)$, all three definitions give us the same cap.

It is well-known that the area of $n$-dimensional sphere of radius $R$ is given by
\begin{equation}\label{sphere*area}
A_n(R)=\frac{2\pi^{n/2}}{\Gamma(n/2)}R^{n-1},
\end{equation}
where $\Gamma(\cdot)$ is the gamma-function. It was shown in \cite{li2011} that the small cap area
with angle $0\leq \theta\leq \pi/2$ is given by
\begin{equation}\label{cap*area}
A_n(R,\theta)=\frac{1}{2}A_n(R)F\left(\sin^2(\theta),\frac{n-1}{2},\frac{1}{2}\right),
\end{equation}
where $F$ is the cumulative distribution function (cdf) of the beta distribution:
$$
F(z,\alpha,\beta)= \frac{\Gamma(\alpha+\beta)}{\Gamma(\alpha)\Gamma(\beta)}\int_0^z u^{\alpha-1}(1-u)^{\beta-1}du,\quad\quad 0<z<1,\quad\alpha,\beta>0.
$$
It is easy to see that in our case the radius of the sphere is $R=\sqrt{n}$ and $\sin^2(\theta)=r/n$. Therefore,
the ratio of the area of the shadow to the area of the sphere is equal to
$$P_n(r)=\frac{1}{2}F\left(\frac{r}{n},\frac{n-1}{2}, \frac{1}{2}\right).$$

In particular, when $n=3$ and $r=1$
$$P_3(1)=\frac{1}{2}\left(1-\sqrt{\frac{2}{3}}\right).$$
Since the shadow caps of 8 unit balls do not overlap the total fraction of blocked light is equal to
$$4\left(1-\sqrt{\frac{2}{3}}\right)\approx .73.$$ Moreover, one can show
that the total fraction of blocked light goes to 0 as $n\to\infty$ for any fixed $r$. Indeed,
we have that
\begin{eqnarray*}
  F\left(\frac{r}{n},\frac{n-1}{2}, \frac{1}{2}\right) & =& \frac{\Gamma(\frac{n+1}{2})}{\Gamma(\frac{n-1}{2})\Gamma(\frac{1}{2})}\int_0^{r/n} u^{\frac{n-1}{2}-1}(1-u)^{\frac{1}{2}-1}du\\
  & \leq& \frac{\frac{n-1}{2}\Gamma(\frac{n-1}{2})}{\Gamma(\frac{n-1}{2})\Gamma(\frac{1}{2})}
    \int_0^{r/n}u^{\frac{n-3}{2}} du \\
    & =& \frac{n-1}{2\sqrt{\pi}} \frac{2}{n-1} \left(\frac{r}{n}\right)^{\frac{n-1}{2}}\\
    & =&\frac{1}{\sqrt{\pi}}\left(\frac{r}{n}\right)^{\frac{n-1}{2}}\,.
\end{eqnarray*}
That is, the ratio of one shadow cap to $A_n(\sqrt{n})$ goes to 0 much faster than
the number of unit balls, $2^n$, goes to $\infty$.

\section{A More Difficult Question}\label{main*question}
It is clear that if the radii of the balls at the vertices of the cube are equal to $\sqrt{n}$,
then the light will be blocked completely. So, here is a new question. Can we
find $\{r_n\}_{n\geq 1}$  such that for  $0<a<1$ we have that $P_n(r_n)\to a$ as $n\to \infty$?
We could not answer this question using geometry. It is not difficult to
figure out an asymptotic behavior of
$F\left(\frac{r_n}{n},\frac{n-1}{2}, \frac{1}{2}\right)$ for a given sequence $\{r_n\}_{n\geq 1}$.
But the issue is that  as soon as $r_n>1$, the balls (and, therefore, the shadow caps)
will overlap. This overlap is not easy to track, because some balls are very close to each other
with the distance of 2 units between their centers, and some pairs have the distance of $2\sqrt{n}$.

However, the question can be successfully addressed with the help of probability theory. The result
is quite surprising. To achieve a probability of $1/2$, the balls must be enormous. Specifically,
their radii must be $\sqrt{(1-2/\pi)n}+o(1)$.  Recall that the radius the cube circumscribing
hypersphere is exactly $\sqrt{n}$. However, a relatively small finite variation of radii
(independent of $n$) will change this probability. For example, if $r_n=\sqrt{(1-2/\pi)n}+1$,
then the fraction of blocked light will be almost 100\% for all sufficiently large $n$.
And it is almost 0\% for $r_n=\sqrt{(1-2/\pi)n}-1$.

As we mentioned above, the solution is probabilistic in nature. We will use two deep  and important but well-known results.
The first one is the celebrated Central Limit Theorem (CLT).
The second statement is about the uniform distribution on  $n$-dimensional unit sphere.
If we generate a vector of $n$ independent identically distributed (i.i.d.) standard normal
random variables  $\mathbf{Y}=(Y_1,\dots,Y_n)$, then $$\mathbf{U}=\frac{\mathbf{Y}}{\sqrt{Y_1^2+\cdots+Y_n^2}}=\frac{1}{\|\mathbf{Y}\|}(Y_1,\dots,Y_n)$$ is  uniformly distributed on surface of $n$-dimensional unit ball \citep{Milman2019TheHO}.

\begin{prop} We call the line associated with vector $\mathbf{Y}$, $\mathbf{y}=\mathbf{Y}t$, $t\in \mathbb{R}$, a {\it random line}. Then, with probability 1, the vertices
$$(\sign(Y_1),\dots,\sign(Y_n))$$
and
$$(-\sign(Y_1),\dots,-\sign(Y_n))$$
are the closest  to and at the same distance from random line $\mathbf{y}=\mathbf{Y}t$, among all $2^n$
vertices of the cube $[-1, 1]^n$.
\end{prop}
\begin{proof}
Since the distance from the origin to any vertex is the same and equal to $\sqrt{n}$, a
vertex $\mathbf{v}$ with the largest absolute value of the cosine of the angle
between $\mathbf{v}$ and $\mathbf{Y}$
will have the smallest distance to the random line. The absolute value of  the cosine  of
this angle is given by
$$\frac{|\mathbf{v}\cdot \mathbf{Y}|}{\sqrt{\|\mathbf{v}\|\|\mathbf{Y}\|}}=\frac{|\mathbf{v}\cdot \mathbf{Y}|}{\sqrt{n\|\mathbf{Y}\|}}.$$
Therefore, the largest value is achieved when all the summands of inner
product $\mathbf{v}\cdot \mathbf{Y}$ have the same sign.
\end{proof}
Now, let $\mathbf{X}=(X_1,\dots,X_n)$ be a vector of  i.i.d. half-normal random variables.
That is, $X_i$ has the same distribution as $|Z|$, where $Z$ is a standard
normal random variable. Then the squared  distance between a random line
(passing through the origin) and the nearest vertex of the cube $[-1,1]^n$,
denoted by $d^2$, is  equal (in distribution) to the squared distance between point $\mathbf{c}=(1,\dots,1)$
and a line $$\mathbf{y}=\mathbf{X}t,\quad t\in\mathbb{R}.$$ This squared distance is given by
$$d^2=\left\|\mathbf{c}-\frac{\mathbf{c}\cdot\mathbf{X}}{\|\mathbf{X}\|^2}\mathbf{X}\right\|^2.$$
Denote
$$s= \mathbf{c}\cdot\mathbf{X}=X_1+\cdots+X_n,$$
and
$$t=\|\mathbf{X}\|^2=X_1^2+\cdots+X_n^2.$$ Note that $s\leq nt$. Then
\begin{eqnarray*}
d^2&=&\left\|\mathbf{c}-\frac{s}{t}\mathbf{X}\right\|^2\\
   &=&\left(1-\frac{s}{t}X_1\right)^2+\dots+\left(1-\frac{s}{t}X_n\right)^2\\
   &=&1-2\frac{s}{t}X_1+\frac{s^2}{t^2}X_1^2+\cdots+1-2\frac{s}{t}X_1+\frac{s^2}{t^2}X_1^2\\
   &=&n-2\frac{s}{t}s+\frac{s^2}{t^2}t\\
   &=&n-\frac{s^2}{t}.
\end{eqnarray*}
The Law of Large Numbers (LLN) immediately gives us that
$$\frac{d^2}{n}=1-\frac{(X_1+\cdots+X_n)^2}{n(X_1^2+\cdots+X_n^2)}=
1-\frac{(X_1+\cdots+X_n)^2}{n^2}\frac{n}{X_1^2+\cdots+X_n^2}\toP 1-\frac{2}{\pi},$$
because $\EE(X_i)=\sqrt{2/\pi}$ and $\EE(X_i^2)=1$. Now,
let us state a general result on the asymptotic behavior of the ratio $s^2/nt$. Note that
this ratio is also the square of the cosine of the angle between a random line $\mathbf{y}=\mathbf{X}t$, $t\in\mathbb{R}$
and the vector $(1,\dots,1)$.
\begin{prop}\label{ratio*normality}
 Let $\{X_i\}_{i\geq 1}$ be a sequence of i.i.d. positive random variables
with $\EE(X_i)=\mu$, $\EE(X_i^2)=1$, $\EE(X_i^3)=a$, and $\EE(X_i^4)=b$. Let
$$s=X_1+\cdots+X_n,$$
and
$$t=X_1^2+\cdots+X_n^2.$$
Then
$$\sqrt{n}\left[\frac{s^2}{nt}-\mu^2\right]\toD N(0,\sigma^2),$$
where $$\sigma^2=\mu^4b+4\mu^2-4\mu^3a-\mu^4.$$
\end{prop}
\begin{proof} The Central Limit Theorem (CLT) tells us that both $(s-\mu n)/\sqrt{n}$
and $(t-n)/\sqrt{n}$ are asymptotically normal. So, first note that
\begin{eqnarray*}
\sqrt{n}\left[\frac{s^2}{nt}-\mu^2\right]
    &=&\frac{1}{n^{3/2}}\left[s^2-\mu^2tn\right]\frac{n}{t}\\
    &=&\frac{1}{n^{3/2}}\left[(s-\mu n)^2+2\mu ns-\mu^2n^2-\mu^2n(t-n)-\mu^2n^2\right]\frac{n}{t}\\
    &=&\frac{1}{n^{3/2}}\left[(s-\mu n)^2+2\mu n(s-\mu n)-\mu^2n(t-n)\right]\frac{n}{t}\\
    &=&\frac{1}{n^{3/2}}\left[(s-\mu n)^2\right]\frac{n}{t}+\frac{1}{n^{1/2}}\left[2\mu(s-\mu n)-\mu^2(t-n)\right]\frac{n}{t}\\
    &=&\frac{1}{n^{3/2}}\left[(s-\mu n)^2\right]\frac{n}{t}+\frac{1}{n^{1/2}}\left[\sum_{i=1}^n\left(2\mu(X_i-\mu)-\mu^2(X_i^2-1)\right)\right]\frac{n}{t}\,.
\end{eqnarray*}
Then, since $\Var(X_i)=1-\mu^2$, $\Var(X_i^2)=b-1$ and $\Cov(X_i,X_i^2)=a-\mu$ by the CLT we
get
$$\frac{1}{n^{1/2}}\left[\sum_{i=1}^n\left(2\mu(X_i-\mu)-\mu^2(X_i^2-1)\right)\right]\toD N(0,\sigma^2).$$
But we also have that
$$\frac{1}{n^{3/2}}\left[(s-\mu n)^2\right]\frac{n}{t}\toP 0$$
and
$$n/t\toP 1.$$
Slutsky's theorem is then applied to complete the proof.
\end{proof}
Now, taking into account that in our case $\EE(X_i)=\sqrt{2/\pi}$, $\EE(X_i^2)=1$,
$\EE(X_i^3)=2\sqrt{2/\pi}$, and $\EE(X_i^4)=3$, by Proposition~\ref{ratio*normality} we obtain that
\begin{equation}\label{our*ratio*normality}
\sqrt{n}\left[s^2/nt-2/\pi\right]\toD N\left(0, \frac{8}{\pi}-\frac{24}{\pi^2}\right).
\end{equation}
Thus, the distribution of squared distance $d^2$ is approximately normal with mean $(1-2/\pi)n$ and
standard deviation $\sqrt{(8/\pi-24/\pi^2)n}$. A side note, the assumption $\EE(X_i^2)=1$
in Proposition~\ref{ratio*normality} is not restrictive. As long as the first four moments
are finite, any distribution can be scaled to satisfy this assumption.

Let $\theta$ be the angle between a random line $\mathbf{y}=\mathbf{Y}t$, $t\in \mathbb{R}$
and the nearest vertex. Then by applying the delta method we obtain the following.
\begin{corr}\label{d*and*theta*normality}
If $n\to\infty$, then
$$\sqrt{n}\left(\cos(\theta)-\sqrt{\frac{2}{\pi}}\right)\toD N\left(0,\frac{\pi-3}{\pi}\right),$$
and
$$d-\sqrt{\left(1-\frac{2}{\pi}\right)n}\toD N\left(0,\frac{2(\pi-3)}{\pi(\pi-2)}\right).$$
\end{corr}
\begin{proof}
Since $\cos^2(\theta)=s^2/nt$, applying the delta method to (\ref{our*ratio*normality})
for the function $g(x)=\sqrt{x}$ we get
$$\sqrt{n}\left[g(\cos^2(\theta))-g\left(2/\pi\right)\right]\toD N\left(0,\left(\frac{8}{\pi}-\frac{24}{\pi^2}\right)\left[g'\left(2/\pi\right)\right]^2\right),$$
and after some algebra we get the first asymptotic result.

Next, from (\ref{our*ratio*normality}) we also have that
$$\sqrt{n}\left[d^2/n-(1-2/\pi)\right]\toD N\left(0, \frac{8}{\pi}-\frac{24}{\pi^2}\right).$$
Therefore,
$$\sqrt{n}\left[g(d^2/n)-g(1-2/\pi)\right]\toD N\left(0, \left(\frac{8}{\pi}-\frac{24}{\pi^2}\right)\left[g'\left(1-2/\pi\right)\right]^2\right),$$
which gives us the second convergence result.
\end{proof}

Corollary \ref{d*and*theta*normality} gives us the answer to the main question of the article, which we formalize now:
\begin{corr}\label{mainresult}
Assume that at every vertex of the cube $[-1,1]^n$ we put a $n$-ball of radius $r_n$. Let $p(r_n)$ be
the probability that a random line will intersect at least one of the balls. Then we have the following.
\begin{enumerate}
  \item If $r_n\leq \sqrt{\alpha n}$, where $\alpha<1-2/\pi$, then  $p(r_n)\to 0$ as $n\to\infty$,
  \item If $r_n\geq \sqrt{\alpha n}$, where $\alpha>1-2/\pi$, then  $p(r_n)\to 1$ as $n\to\infty$,
  \item If $r_n= \sqrt{(1-2/\pi) n}+z\sqrt{2(\pi-3)/(\pi(\pi-2))}$, where  $z\in \mathbb{R}$, then  $p(r_n)\to \Phi(z)$ as $n\to\infty$,
  where $\Phi$ is the cdf of the standard normal distribution.
\end{enumerate}
\end{corr}
The first two statements follow from the LLN, and the last one is a consequence of the CLT.

\section{Concentration of Measure}\label{sec:com}

Let us discuss our findings. For large $n$, to block a fraction of the light, the balls at the
vertices of the cube $[-1,1]^n$ must be large, with a radius of $O(\sqrt{n})$. Moreover, a small change in their
radii, just by $O(1)$, will have a significant impact. We did state that you have to be cautious with
applying 3-D intuition to high-dimensional space. But on the other hand, it is human nature to
try to find some big-picture explanations. The fact that the balls are $O(\sqrt{n})$ in size is
not that surprising. After all, the cube $[-1,1]^n$ is a large object, and the distance between
opposite vertices is $2\sqrt{n}$. The second result is more difficult to accept. We believe,
however, it can be ``explained'' with the help of the phenomenon known as ``concentration
of measure''.

There is a deep connection between the concentration of measure and the Law of Large Numbers (LLN)
for independent random variables. This connection is the main topic of \citeauthor{talagrand1996}'s
(\citeyear{talagrand1996}) seminal paper. Let us first illustrate the main idea using a simple
example.

Consider again the cube $[-1,1]^n$, and a discrete \textit{uniform} distribution on the set
of $2^n$ vertices, so that the probability of randomly drawing each of the $2^n$ vertices
of the cube is $2^{-n}$. Note that all the vertices also belong to $n$ dimensional
sphere of radius $\sqrt{n}$. 

Now, let us fix a vertex $\mathbf{p}$, which we will
call a \textit{pole}, and its opposite pole is obtained by multiplying $\mathbf{p}$ by $-1$.
We define the $k$-th latitude $$L_k=\{\mathbf{x}\in \{-1, 1\}^n ~:~ ||\mathbf{x} - \mathbf{p}||^2 = 4k\}\,,$$
for $k=0,1,\ldots,n$. In other words, the latitude $L_k$ is the set of all vertices that
disagree with $\mathbf{p}$ in exactly $k$ coordinates.

When $n$ is even,  $L_{n/2}$ can be called ``\textit{equator}'', because it
is the set of points that are equidistant to $\mathbf{p}$ and $-\mathbf{p}$.
It is easy to see that the points in the equator are {\it perpendicular} to the pole $\mathbf{p}$.
Indeed, $\mathbf{x}\in L_{n/2}$ if exactly $n/2$ of its coordinates are the same as in $\mathbf{p}$.
Therefore,  $\mathbf{p}\cdot\mathbf{x}=0$, that is, the cosine of the angle between $\mathbf{p}$
and $\mathbf{x}$ is 0.

A randomly chosen vertex falls in the $k$-th latitude relative to the pole
with probability $p(k)$, which follows the binomial distribution with $n$ trials
and probability of success 1/2. This is so because, being in the $k$-th latitude is equivalent
to choosing $k$ coordinates in $\mathbf{p}$ and multiplying them by $-1$.
If we denote the latitude relative to $\mathbf{p}$ in which a random vertex falls by $K$,
then the expected value of the latitude is $\EE(K)=n/2$ and the variance is $\Var(K)=n/4$.
Therefore, Bernoulli's LLN for the binomial distribution tells us that for
any (small) $\epsilon>0$ and sufficiently large $n$, with probability close to 1
the latitude of the random vertices  will be within $\epsilon n$
distance from $n/2$. Of course, similarly to our Corollary~\ref{mainresult}, the CLT
will give us an even more precise statement.

Thus, for large $n$, the thin central slab of the sphere $S^{n}(\sqrt{n})$  contains
almost all random vertices. Alternatively, we can say that almost all random vertices lie outside
two large caps centered at the two opposite poles $\mathbf{p}$ and $-\mathbf{p}$.
Considering that our choice of poles is arbitrary, this statement is also correct for any other
central slab. 

Now, let us explain how to reach the same conclusion using the concentration of measure of the product space.
Note that the uniform distribution over $2^n$ vertices is indeed an $n$-product space of the discrete
probability space consisting of two points $\{-1,1\}$ with a probability of $1/2$ for each. With a bit of effort,
one can show that Proposition 2.1 of \cite{talagrand1996} is directly applicable
to our situation. Essentially, it states that {\it if we select any subset $A$  of the vertices such that $\PP(A) \geq 1/2$,
then most of the remaining vertices will be close to  $A$}. This implies that if we take a hemisphere which contains the pole $\mathbf{p}$
(including the equator if $ n$  is even), then with high probability, the rest of the vertices will be close to this set,
meaning they will belong to a thin slab right under the hemisphere.
Since the same is true for the opposite hemisphere with pole $-\mathbf{p}$, we obtain the same result as before:
with high probability, a random vertex belongs to a thin central slab.
However, this ``proof''  is based on the concentration of measure, not the Law of Large Numbers.
Also note that, as before, the same can be derived for hemispheres with different opposing poles.

So, how can the concentration of measure phenomenon  be linked to our fraction of light story? Think
about the uniform probability measure over the $n$-dimensional sphere that  circumscribes the cube $[-1,1]^n$.
The union of all shadow caps from the balls at the vertices is a subset of the sphere.
For large $n$, when the ball radii are about the size of $\sqrt{(1-2/\pi) n}$, the union of all
shadows has a measure close to $1/2$. Therefore, if you keep in mind the concentration of
measure phenomenon,  it seems plausible that increasing the cap
sizes by a constant ($O(1)$) will create almost a complete cover of the entire $n$-dimensional sphere. One, of course, can note
that the $n$-dimensional sphere with uniform distribution is not a product space. But if we recall how one can
generate a random  point on the $n$-dimensional sphere with the uniform distribution, then we can see that for large $n$
any two coordinates of the point are almost independent. The only common factor they have is
the normalizing square root of the sum of squares that for large $n$ by the LLN is almost a constant.
Since independence and product space are two closely related mathematical concepts, it is
not surprising that for large $n$ the $n$-dimensional sphere with the uniform distribution has the
concentration of measure property.

\section{Discussion}\label{sec:disc}
Is high dimensional data a blessing or a curse? Clearly, it presents significant challenges, both theoretical
and computational. It is also impossible to visualize or to obtain geometrical interpretations for such data.
It is tempting to assume that our intuition in the two dimensional Euclidean space extends to higher dimensions,
but this is not so. However, through probability models one can derive asymptotic results and obtain
a geometric intuition, which is often surprising, if not  counterintuitive. In particular, high-dimensional
data tends to concentrate at very specific regions which depend on our choice of coordinate
system and reference points. For example, relative to a point of origin, multivariate normal distribution
in $\mathbb{R}^n$ concentrates very close to a sphere of radius $\sqrt{n}$, and if one chooses two
points as opposing poles on the sphere, then the data concentrates near the equator relative to the poles.

We highlight the deep connection between
concentration of measure and the law of large numbers for independent random variables.
Consider what the discrete example in the previous section means in the context of receiving a binary-coded message, such as
ones transmitted by GPS satellites. If there is no prior
information and we do not know what message is expected, each sequence of length $n$
is equally likely. But, if we expect a specific message
$\mathbf{p}$ (e.g., the identifier of a specific satellite), then the vast majority of
random messages will be nearly perpendicular to $\mathbf{p}$, making it
very easy to detect them as noise. In contrast, a message with few errors will be in a latitude
far enough from the equator, and close to the cap centered at $\mathbf{p}$.
This example demonstrates the fact that  high dimension data is actually a blessing.
Indeed, concentration of measure leads to what is called in the literature  ``\textit{the blessing of dimensionality}'',
a term coined by \cite{donoho2000high}.

In the discrete example discussed in Section \ref{sec:com} the random vectors are drawn
uniformly from the set of $2^n$ cube vertices.
In our paper we assume that an $n$-dimensional random vector comes from a standard
multivariate normal distribution. The support is the entire Euclidean space, but most of
the drawn data are concentrated very close to the  sphere of radius $\sqrt{n}$. This is so because,
although the expected location of a random univariate normal variable $x$ is 0, its
expected length is $1$ (the length is $|x_i| = \sqrt{x_i^2}$, which has a half-normal (or a chi) distribution.)
The sum of $n$ independent squared standard normal random variables has a chi-squared distribution
with $n$ degrees of freedom, with mean $n$. Thus, using the CLT, one can show that random
multivariate normal
data is concentrated close to the sphere of radius $\sqrt{n}$. The origin of this famous result is discussed by
\cite{Milman2019TheHO} which highlights Lévy's contributions.
\citeauthor{Levy}'s (\citeyear{Levy}) motivation was a paper by \cite{Borel1914} who made a remark
about geometric interpretation of the
law of large numbers, but the result was even known a couple of years prior, to Poincaré  (\citeyear{Poincare1912}).









\bibliographystyle{Chicago}

\bibliography{SpaceOddityArxiv1}
\end{document}